\documentclass[10pt,a4paper]{amsart}
\usepackage[english]{babel}
\usepackage{amssymb,xypic,amscd,amsthm,stmaryrd, mathtools, amsmath}

\CompileMatrices
\usepackage[T1]{fontenc}
\newtheorem{defn}{Definition}
\newtheorem{thm}[defn]{Theorem}
\newtheorem{cor}[defn]{Corollary}
\newtheorem{lem}[defn]{Lemma}
\newtheorem{prop}[defn]{Proposition}

\theoremstyle{plain}
\newtheorem{rem}[defn]{Remark}
\theoremstyle{remark}
\newtheorem{exam}{Example}
\numberwithin{equation}{section} \numberwithin{defn}{section}
\usepackage{multicol}
\usepackage{float}
\usepackage{xfrac}
\usepackage{faktor}

\usepackage{tikz-cd}

\newcommand\ed{\operatorname{End}}
\newcommand{\Img}{\operatorname{Im}}

\newcommand{\Coker}{\operatorname{Coker}}
\newcommand\Ker{\operatorname{Ker}}
\newcommand\aut{\operatorname{Aut}}

\newcommand{\f}{\varphi}

\newcommand{\N}{\mathbb{N}}

\newcommand\Id{\operatorname{Id}}
\begin{document}

\title{ k[x]-modules and Core-Nilpotent endomorphisms}

\author{ Diego Alba Alonso* }
\author{Javier Sánchez González**}

\address{Departamento de Matem\'aticas, ETSII, Universidad de Castilla-La Mancha, 13071 Ciudad Real, Spain}
 \email{ (*) Diego.Alba@uclm.es}
 \email{ (**) Javier.SGonzalez@uclm.es}

\thanks{This work was supported by {\it Agencia Estatal de Investigación} (Spain) through grant PID2023-151823NB-I00.}

\begin{abstract}
Core-nilpotent endomorphisms over an arbitrary vector space form the largest subset of the ring of endomorphisms over that arbitrary vector space which admit a decomposition as sum of two endomorphisms satisfying the analogous properties as the well known core-nilpotent decomposition of matrices. In this paper we present a new description of core-nilpotent endomorphisms using the $k[x]-$module structure they define in the base vector space. Moreover, our approach provides us with a ``new'' generalized inverse that restricts to the well known Drazin inverse under certain conditions. Similarly, we present a generalized core-nilpotent decomposition for endomorphisms over arbitrary vector spaces.
\end{abstract}

\maketitle

\bigskip

\setcounter{tocdepth}1

\tableofcontents
\bigskip

\medskip

\bigskip 

$k[x]-$module, Index, Drazin Inverse, Localization.
\section{Introduction.}

Let us consider a square matrix $A$ with entries in an arbitrary field $k.$ If we denote by $\mathrm{rk (A)}$ the rank of the matrix $A,$ the index $A,$ which is denoted by $i(A);$ is the smallest integer satisfying that $\mathrm{rk (A^{i(A)})}=\mathrm{rk (A^{i(A)+1})}.$ In \cite[Theorem 2.2.21]{Ind}, it was shown that any square matrix $A$ can be written as the sum of two other matrices which are usually denoted as $A_1$ and $A_2$ such that: \begin{itemize}
\item $i(A_1)\leq 1;$
\item $A_2$ is nilpotent;
\item $A_1\cdot A_2=0=A_2\cdot A_1.$
\end{itemize}
This decomposition is unique and it is named as the ``core-nilpotent'' decomposition of matrix $A.$ Matrices $A_1$ and $A_2$ are called the core and nilpotent parts of $A$ respectively. Applications of the core-nilpotent decomposition vary from solving differential-algebraic equations, to the study of generalized inverses and matrix partial orders, analysing matrices associated to graphs (like the adjacency matrix of a tree) and there are even generalizations of this decomposition such as the generalized core-nilpotent decomposition. See, among others, \cite{Gen CN}, \cite{Ind} and \cite{Graphs}. \\
In 2021, \cite{Fpa-CN}, a natural problem was studied related to this decomposition. Precisely, a notion of index that generalizes the index of a finite square matrix was introduced to define core-nilpotent endomorphisms on arbitrary vector spaces. These are the largest class of endomorphisms that admit a decomposition which is equivalent to the homonym one for the case of square matrices. In particular, all endomorphisms that have Drazin inverse were characterized.

In this paper, we propose a new algebraic formulation of this problem, based on localization of $k[x]$-modules, which naturally leads to the notion of core-nilpotent endomorphism and, in fact, to the slightly more general notion of \textit{pointwise core-nilpotent endomorphism}. Our approach clarifies the exact role that each hypothesis plays in the theory as far as existence and uniqueness of the decomposition are concerned. In the process, we also introduce a generalization of the Drazin inverse whose existence characterizes pointwise core-nilpotent endomorphisms and that restricts to the ordinary Drazin inverse in the core-nilpotent case. Both of these ``pointwise'' notions relax the nilpotency conditions seen in the classical literature, whicn allows us to treat well known operators ---such as the formal derivative operator for polynomials--- within the boundaries of our theory.

\medskip

\section{Preliminaries} \label{s:pre}
This section is included for the sake of completeness.\\

Let $k$ be a field, let $V$ be an arbitrary $k-$vector space (in general, infinite dimensional) and let $\f\colon V\to V$ be an endomorphism.

\begin{defn}\cite[Definition 3.1]{Fpa-CN}\label{D: index0}
An endomorphism $\f \in \ed_k(V)$ has index $0$ when it is an automorphism, $\f \in \aut_k(V).$
\end{defn}

\begin{defn}\cite[Definition 3.2]{Fpa-CN}\label{D: Index}
We will call index of an endomorphism $\f \in \ed_k(V)$ to the smallest positive integer $m$ for which $\Ker(\f^m)=\Ker(\f^{m+1})$ and $\Img(\f^m)=\Img(\f^{m+1}).$
\end{defn}

\begin{defn}\cite[Definition 3.4]{Fpa-CN}\label{D: CN}
We say that an endomorphism $\f \in \ed_k(V)$ is core-nilpotent (or it is a CN-endomorphism) when there exists two endomorphisms $\f_1,\f_2 \in \ed_k(V)$ such that: \begin{itemize}
\item $\f=\f_1+\f_2;$
\item $i(\f_1)\leq 1;$
\item $\f_2$ is nilpotent;
\item $\f_1\circ \f_2=0=\f_2\circ \f_1.$
\end{itemize}
\end{defn}

\begin{thm}\cite[Theorem 3.6]{Fpa-CN}\label{T: Charact CN endos FP}
If $\f \in \ed_k(V)$ then, the following conditions are equivalent: \begin{itemize}
\item $\f$ is a CN-endomorphism.
\item $\Ker(\f^m)=\Ker(\f^{m+1})$ and $\Img(\f^m)=\Img(\f^{m+1})$ for certain $m\in \N .$
\item $V=\Ker(\f^m)\oplus \Img(\f^m)$ for a certain $m\in \N .$
\item There exists a unique decomposition $V=W_{\f}\oplus U_{\f},$ where $W_{\f}$ and $U_{\f}$ are $\f$-invariant subspaces of $V,$ $\f_{_{\vert W_{\f}}}\in \aut_k(W_{\f})$ and $\f_{\vert U_{\f}}$ is nilpotent.
\end{itemize}
\end{thm}

\begin{lem}\cite[Lemma 3.13]{Fpa-CN}\label{L: IndexSiiCN}
The index of an endomorphism $\f \in \ed_k(V)$ exists if and only if $\f$ is a CN-endomorphism.
\end{lem}

Given a matrix $A\in \text{Mat}_{n\times n}(k),$ M.P Drazin proved, in \cite{Dra}, the existence of an unique matrix that he denoted $A^D\in \text{Mat}_{n\times n}(k)$ which later was named in his honour as Drazin inverse, satisfying the following equations: \begin{align}
\begin{split}
A^D\cdot A\cdot A^D & = A^D;\\
A^D\cdot A & = A\cdot A^D;\\
A^{m+1}\cdot A^D & = A^m, \text{ for }m=i(A).
\end{split}
\end{align}
S.L. Campbell, in 1976, approached firstly the study of the Drazin inverse for ``infinite-dimensional'' complex matrices in \cite{Camp}. Considering an arbitrary $k-$vector space and a finite potent endomorphism, F. Pablos Romo proved the existence and uniqueness of the Drazin inverse of a core-nilpotent endomorphism (which is defined in an analogous way as in the matrix case) in \cite[Proposition 4.2]{Fpa-CN}. Moreover, the following result was proven:

\begin{thm}\cite[Theorem 4.8]{Fpa-CN}\label{T: DrazinSiiCN} An endomorphism $\f \in \ed_k(V)$ has a Drazin inverse if and only if $\f$ is a core-nilpotent endomorphism.
\end{thm}

\section{k[x]-modules and core-nilpotent endomorphisms}\label{S: kx-modules and cn endo}
Let $V$ denote an arbitrary $k-$vector space. For every linear operator $\f \colon V \to V$ one has the following two chains of subspaces (which are non-isomorphic): \begin{align}\label{eq 1a}
0\subseteq \Ker(\f)\subseteq \Ker(\f^2)\subseteq...\subseteq \Ker(\f^m)\subseteq...,
\end{align}
which, in virtue of the First Isomorphism Theorem, is equivalent to
\begin{align*}
V\overset{\f}{\to}\Img(\f)\overset{\f}{\to} \Img(\f^2)\overset{\f}{\to}...\overset{\f}{\to}\Img(\f^m)\overset{\f}{\to}...;
\end{align*}
and on the other hand
\begin{align}\label{eq 2a}
...\subseteq\Img(\f^m)\subseteq...\subseteq \Img(\f^2)\subseteq \Img(\f)\subseteq V,
\end{align}
which is equivalent to the following chain of quotients (by definition of cokernels):
\begin{align*}
...\to \Coker(\f^m)\to ...\to \Coker(\f^2)\to \Coker(\f)\to 0.
\end{align*}

\begin{rem}

We shall remark that the chains presented in \eqref{eq 1a} and \eqref{eq 2a}, necessarily stabilize when the vector space is of finite dimension. On the contrary, when the vector space is infinite dimensional, these chains may not stabilize, or in fact, one may stabilize and the other may not. Let us illustrate this fact with the following examples:

\end{rem}

\begin{exam}Let us consider the following vector space with infinite countable dimension $V=\underset{i\in \N}{\bigoplus}\langle v_i \rangle$ over an arbitrary vector space $k,$ and let us define $\f \colon V \to V$ by linearity from the expression:

$$\f(v_{i}) = \left \{ \begin{array}{ccl} 0 & \text{ if } & i=1 \\  v_{i-1} & \text{ if } & i\geq 2 \end{array} \right . .$$ Note that $\Img(\f^m)=V$ and $\Ker(\f^m)=\langle v_1,v_2,\dots , v_m \rangle $ for every $m\in \N .$ Therefore, the chain formed by the inclusions of the images, \eqref{eq 2a}, stabilizes for $m=1$ and the one formed by the kernels do not stabilize.\\ Furthermore, this example highlights a notable difference between linear algebra over infinite dimensional vector spaces and linear algebra over finite dimensional vector spaces (matrix theory), there exists endomorphisms that are surjective and are not injective (one can also find examples for the opposite). 
\end{exam}

\begin{exam}\label{E: ex 2}
Let us now include an example of an endomorphism over a $k-$vector space in which the chain presented in \eqref{eq 1a} stabilizes and another one in which it does not. Let $$\begin{array}{rccl}
\f_1 \colon & k[x]\oplus k[x]/x^n  & \rightarrow & k[x]\oplus k[x]/x^n \medskip \\
& (p(x),\overline{q(x)}) & \mapsto & (x\cdot p(x), \overline{x\cdot q(x)}).
\end{array} $$
Note that: \begin{itemize}
\item $\Ker(\f_1)=0\oplus \langle\overline{x^{n-1}} \rangle\subseteq k[x]\oplus k[x]/x^n;$
\item $\Ker(\f_1^2)=0\oplus \langle\overline{x^{n-2}}, \overline{x^{n-1}} \rangle\subseteq k[x]\oplus k[x]/x^n;$\\
\vdots 
\item $\Ker(\f_1^n)=0\oplus \langle k[x]/x^n \rangle\subseteq k[x]\oplus k[x]/x^n;$
\end{itemize}
and therefore the chain stabilizes and: \begin{equation}\label{eq: ejeunionnucleos1}
\underset{i=1}{\bigcup^n}\Ker(\f_1^i)=0\oplus k[x]/x^n.
\end{equation}

On the other hand, let us consider the following endomorphism: $$h_x \colon \underset{n\geq 1}{\bigoplus^{\infty}} (k[x]/x^n) \to \underset{n\geq 1}{\bigoplus^{\infty}} (k[x]/x^n),  $$ again, defined as the homothecy given by multiplication by $x$ at each component. A straightforward calculation shows that, setting $x^0=1$ and $x^j=0$ for any $j<0$: 
\begin{itemize}
\item $\Ker(h_x)=\bigoplus_{m\geq 0}\langle \overline{x^{m}}\rangle\subseteq \bigoplus_{n\geq 1}k[x]/x^n;$
\item $\Ker(h_x^2)=\bigoplus_{m\geq 0}\langle \overline{x^{m}}, \overline{x^{m-1}}\rangle\subseteq \bigoplus_{n\geq 1}k[x]/x^n;$\\
\vdots 
\item $\Ker(h_x^i)=\bigoplus_{m\geq 0}\langle \overline{x^{m}}, \overline{x^{m-1}}, ..., \overline{x^{m-i}}\rangle\subseteq \bigoplus_{n\geq 1}k[x]/x^n;$\\
\vdots
\end{itemize}  
where $\overline{x}$ denotes the class of $x$ in the corresponding quotient. It is clear that the chain of inclusions $$\Ker(h_x)\subset \Ker(h_x^2)\subset \dots \subset \Ker(h_x^i)\subset \dots $$ does not stabilize and, in fact,
 \begin{align*}
\underset{i=1}{\bigcup^\infty}\Ker(h_x^i) =\underset{n\geq 1}{\bigoplus^{\infty}} (k[x]/x^n)
\end{align*}
This states that every vector is anihilated by some power of $h_x$, but note that $h_x$ is not nilpotent.

\end{exam}

\subsection{On chain conditions.}\label{ss: Chain Conditions}

Before continuing with the theory of modules, we shall point out something. The reader may have noticed that the third item presented in the characterization of Theorem \ref{T: Charact CN endos FP} bears certain resemblance to the decomposition expressed in the well known Fitting Decomposition Theorem. To wit:

\begin{thm}(\textbf{Fitting Decomposition Theorem}.)\label{T: FDT}
Let $R$ be a ring and $M$ a finite-length module over $R.$ Then, for any $\f \in \ed(M),$ the endomorphism ring of $M,$ there is a positive integer $m$ such that $$M=\Ker(\f^m)\oplus \Img(\f^m).$$
\end{thm}

It is well known that considering an endomorphism $\f \colon V\to V$ is equivalent to fix a $k[x]$-module structure in $V$ compatible with the vector space structure. This is, giving a morphism of $k$-algebras $\Phi\colon k[x]\to \ed_k(V)$ where: $$x\cdot v\equiv \Phi(x)(v):=\f(v),$$ for any $v\in V.$

Furthermore, it is also a classical result that a module is of finite length if and only if it is Artinian and Noetherian. In the case of a module over a field, a vector space, it is equivalent to be finite dimensional as a vector space than to be a module of finite length (see \cite[Proposition 6.10]{Ati}). 
\\Moreover, in the proof of Fitting Decomposition Theorem, the hypothesis of being Artinian and Noetherian are used to ensure that the chains presented in \eqref{eq 1a} (Noetherian) and \eqref{eq 2a} (Artinian) stabilize respectively. We shall devote the rest of the section to show that these conditions can be relaxed. Roughly speaking, we do not need to take control over all of the chains of submodules of our module to ensure that we have the decomposition and neither we shall impose any condition which is relative to the dimension of $V$ as a vector space. Notice that core-nilpotent endomorphisms were introduced for arbitrary vector spaces and do have the aforementioned decomposition. This is why the theory related to the decomposition presented in Theorem \ref{T: FDT} for the case of the module of endomorphisms over an arbitrary vector space shall be approached with other tools that enable us to establish a sharpened characterization for its existence in the language of $k[x]-$modules.\smallskip \\

Given an arbitrary $k-$vector space $V,$ and an endomorphism $\f \colon V \to V,$ the $k[x]-$module structure enables us to consider the localization in the multiplicative system $S_{x}=\{ 1,x,x^2,x^3,\dots, x^n,\dots\}$ that henceforth will be denoted as $$ V_{x}:=\faktor{(V\times S_x)}{\sim},$$ where $(v,x^n)\sim (v',x^m)$ if there exists an $x^j\in S_{x}$ such that $x^j(x^m\cdot v-x^n\cdot v')=0.$ Therefore, considering the localization morphism $\phi \colon V \to V_x$, one has the following exact sequence of $k[x]-$modules, and, in particular, of $k-$vector spaces: 
\begin{equation}\label{eq: Suc Exac Loc} 
0 \to \Ker(\phi) \to V \stackrel{\phi}{\to} V_x.
\end{equation}
As short exact sequences of vector spaces split, or in other words, as any vector subspace of a given vector space admits a supplementary subspace, there exists an isomorphism of $k-$vector spaces \begin{equation}\label{eq: IsoKespvect}
 V\simeq \Ker(\phi) \oplus \Img(\phi) .
\end{equation}
The following section that contains the main body of this article, is dedicated to answer two problems related to this isomorphism of $k-$vector spaces. \begin{itemize}
\item Giving a description of both, $\Ker(\phi)$ and $\Img(\phi)$ in terms of subspaces associated to the fixed endomorphism $\f \colon V \to V.$\smallskip
\item Studying the conditions under which the isomorphism of vector spaces presented in \eqref{eq: IsoKespvect} can be chosen to be an isomorphism of $k[x]-$modules.
\end{itemize}

Before continuing and maintaining the notations, recall that in general, $\Img(\phi)$ is not a sub--$k[x]$-module of $V$ but of $V_x.$ Therefore, we shall point out that the second statement is directly linked to realize $\Img(\phi)$ as a sub--$k[x]$-module of $V.$

\subsection{Study of the localization morphism}\label{ss: Study of Localization M}
Let $V$ be an arbitrary $k-$vector space and let us recall that we denote by $V_x$ the localization in the multiplicative system $S_{x}=\{1,x,x^2,\dots,x^n,\dots\}.$

\begin{prop}\label{P: Nuc y Loc}
Let $\f \in \ed_k(V).$ Let $\phi \colon V \to V_x$ be the localization morphism in the multiplicative system $S_x$. The following statements hold: 

\begin{itemize}
\item[I.)] $\Ker(\phi) =\underset{n\geq 0}{\bigcup}\Ker(\f^n) $ , with $\f^0=\Id\in \ed_k(V), $ being $\Id$ the identity morphism.  
\item[II.)] $V_x=0$ if and only if each $v\in V$ is annihilated by some power $\f^{n(v)}$ with $n(v)\in \N.$
\item[III.)] $\varphi$ is injective if and only if $\phi\colon V\to V_x$ is injective.
\end{itemize}
\end{prop}

\begin{proof}
Let us begin by proving $I).$ By construction of the localization, a fraction (i.e. an element of the localization) $\frac{v}{x^i}\in V_x$ is $0$ when there exists some $x^j\in S_x ,$ with $j\geq 0,$ such that $x^j\cdot v=0.$ By the structure of $k[x]-$module this condition is equivalent to say that a fraction is $\frac{v}{x^i}=0$ when $v\in \Ker(\f^j)$ for certain $j\geq 0.$\\
Statement $II)$ is a direct consequence of the aforementioned relation between the equivalent class of $0$ in $V_x$ and the kernels of the powers of $\f.$ The third statement is obvious.
\end{proof}
We shall highlight that $II)$ of the previous proposition occurs, for example, if $\f$ is a nilpotent endomorphism.
\begin{defn}\label{D: pw-nilpo}
We way that an endomorphism $\f\in \ed_k(V)$ is \textit{pointwise nilpotent} if $V_x=0$, i.e. each $v\in V$ is annihilated by some power $\f^{n(v)}$ with $n(v)\in \N.$
\end{defn}

Let us now study the image of the localization morphism. Let us consider the following inductive system of vector spaces: $$ V \stackrel{\f}{\to} \Img(\f) \stackrel{\f}{\to} \Img(\f^2) \stackrel{\f}{\to} \dots \stackrel{\f}{\to} \Img(\f^n) \stackrel{\f}{\to} \dots .$$
\begin{prop}\label{P: Colim Img}
Let $\f \in \ed_k(V)$ be an endomorphism over a $k-$vector space. Let $\phi \colon V \to V_x$ be the localization morphism in the multiplicative system $S_x.$ 
There exists a natural isomorphism of $k[x]-$modules: $$\underset{i\geq 0}{\varinjlim} \Img(\f^i) \simeq \Img(\phi), $$ taking direct limit within the inductive system described previously. 
\end{prop}
\begin{proof}
For every $j>i,$ there exists commutative diagrams:  
\begin{align*}
\xymatrix{ 
0 \ar[r] & \Ker(\f^i) \ar[d] \ar[r] & V\ar[d]^{\Id}\ar[r]^{\f^i}& \Img(\f^i) \ar[r] \ar[d]^{\f^{j-i}} & 0 \\
 0 \ar[r] & \Ker(\f^j) \ar[r] & V \ar[r]^{\f^j}& \Img(\f^j) \ar[r] & 0  
}
\end{align*} where the vertical arrow in the left is the natural inclusion. In virtue of Proposition \ref{P: Nuc y Loc}, the following identifications hold: $$\underset{i\geq 0}{\varinjlim} \Ker(\f^i) \simeq \underset{i \geq 0}{\bigcup} \Ker(\f^i) =\Ker(\phi). $$  To conclude, it is sufficient to recall that filtered inductive limits of modules are exact. Hence, the result is deduced using the ``Fifth's Lemma'' in the following exact sequences: \begin{align*}
\xymatrix{ 
0 \ar[r] & \Ker(\phi) \ar[d] \ar[r] & V\ar[d]\ar[r]^{ \underset{i\geq 0}{\varinjlim}(\f^i)}& \underset{i \geq 0}{\varinjlim}  \Img(\f^i) \ar[r] \ar[d] & 0 \\
 0 \ar[r] & \Ker(\phi) \ar[r] & V \ar[r]^{\phi}& \Img(\phi) \ar[r] & 0 . 
}
\end{align*} 
\end{proof}

\begin{lem}\label{L: Kerphi f inv}
Let $\f \in \ed_k(V)$ be an endomorphism over a $k-$vector space. Let $\phi \colon V \to V_x$ be the morphism of localization in the multiplicative system $S_x.$ The subspace $\Ker(\phi) \subseteq V$ is $\f-$invariant.
\end{lem}
\begin{proof}
The kernel of a morphism of $k[x]-$modules is always a sub--$k[x]$-module.
\end{proof}

\begin{rem}
By construction, while multiplication by $x$ is an isomorphism on $V_x$, it is not necessarily an isomorphism on a sub-$k[x]$-module such as $\Img(\phi)$.
\end{rem}

\begin{exam}Let us go back to the endomorphisms $\varphi_1$ and $\varphi_2$ of Example \ref{E: ex 2}. If we now consider the localization, as localizing commutes with taking direct sums, one has that: $$\phi_{\f_1} \colon k[x]\oplus k[x]/x^n \to (k[x])_x\oplus (k[x]/x^n)_x, $$ and moreover, the multiplicative system $S_x$ intersects the annihilator of $k[x]/x^n,$ therefore: $$\phi_{\f_1} \colon k[x]\oplus k[x]/x^n \to (k[x])_x\oplus 0. $$ Hence, one obtains that $\Ker(\phi_{\f_1})=0\oplus k[x]/x^n,$ which is coherent with the calculations showed in \eqref{eq: ejeunionnucleos1} and the result presented in Proposition \ref{P: Nuc y Loc}.

As for the endomorphism $h_x$ in Example \ref{E: ex 2}: if we consider the localization morphism, using that localizing commutes with taking direct sums $$\phi_{h_x} \colon \underset{n\geq 1}{\bigoplus^{\infty}} (k[x]/x^n) \to \underset{n\geq 1}{\bigoplus^{\infty}} (k[x]/x^n)_{x}, $$ and similarly to the previous example, the multiplicative system $S_x$ intersects the annihilator of the module and therefore: $$\phi_{h_x} \colon \underset{n\geq 1}{\bigoplus^{\infty}} (k[x]/x^n) \to 0. $$ Hence, $$\Ker(\phi_{h_x})= \underset{n\geq 1}{\bigoplus^{\infty}} (k[x]/x^n)$$ which agrees with the previous calculation and with Proposition \ref{P: Nuc y Loc}. 
\end{exam}
\begin{exam}
Consider the derivative operator in one variable 
$$D\colon k[x]\to k[x]$$
with $D(p(x))=p'(x)$. This operator clearly satisfies that $\Ker(\phi_D)=k[x]$.
\end{exam}

Now, consider the exact sequence
$$0\to\Ker(\phi)\to V\to\Img(\phi)\to 0,$$
whose components we have described in the previous Lemmas. It is a well known fact of commutative algebra that sections of $\sigma\colon V\to \Img(\phi)$ are in correspondence with isomorphisms
$$V\simeq \Ker(\phi)\oplus \Img(\phi)$$ In terms of actual subspaces of $V$, this translates into the following:
\begin{thm}[\textbf{Properties and uniqueness of the splitting}]\label{T: general}
If the morphism $\phi$ admits a section of $k[x]$-modules, then there is a decomposition $V=U\oplus W$ with $U, W\subseteq V$ being $\varphi$-invariant subspaces such that $U_x=0$ and so that $W\to W_x$ is injective (i.e. $\f_{|W}$ is injective). 

Additionally, $W\overset{\sim}{\to}W_x$ is an isomorphism ($\varphi_{|W}$ is an isomorphism) if and only if $\phi\colon V\to V_x$ is surjective and, in this case, the decomposition is unique and the converse holds: if $V=U\oplus W$ with
 $U, W\subseteq V$ being $\varphi$-invariant subspaces, $U_x=0$ and $W\simeq W_x$, then $\phi$ admits a (unique) section of $k[x]$-modules.
\begin{proof}
Let $\sigma\colon\Img(\phi)\to V$ be a section of $k[x]$-modules of $\phi$. We obtain a decomposition as a direct sum of subspaces
$$V=U\oplus W$$
with $U=\Ker(\phi)$ and $W=\sigma(\Img(\phi))$. By construction we have that $U_x=0$ and, since multiplication by $x$ is invertible at $V_x,$ also by construction, it is injective when restricted to the sub-$k[x]$-module $\Img(\phi)\subseteq V_x$. By Proposition \ref{P: Nuc y Loc}, this implies that $\Img(\phi)\to \Img(\phi)_x$ is injective and, since $\sigma\colon\Img(\phi)\overset{\sim}{\to}W\subseteq V$ identifies the abstract space $\Img(\phi)$ with its image $W$ (as $k[x]$-modules), we conclude that $W\to W_x$ is injective as well. 

Now it is clear that $W\to W_x$ is an isomorphism if and only if $\Img(\phi)\to \Img(\phi)_x$ is an isomorphism, but one also has that the section $\sigma\colon \Img(\phi)\hookrightarrow V$ induces an isomorphism $\Img(\phi)_x\simeq V_x$ by exactness of localization, so the commutative diagram:
\begin{align*}
\xymatrix{ 
V\ar[r]^{\phi}\ar[d]_{\phi} & \Img(\phi)\ar[d]^{\overline{\phi}}\\
V_x&\ar[l]^{\sim\,\,\,\,\,\,}  \Img(\phi)_x
}
\end{align*}
where the top arrow is surjective and the right arrow $\overline{\phi}$ is injective, concludes that $W\to W_x$ is an isomorphism if and only if $V\to V_x$ is surjective.

Now, assuming that $V\to V_x$ is indeed surjective, we prove that, if $V=U\oplus W$ with $U_x=0$ and $W\simeq W_x$, then $U=\Ker(\phi)$ and $W\simeq \Img(\phi)\simeq V_x,$ this is, uniqueness up to isomorphism as abstract spaces. Due to exactness of localization, we may consider the commutative diagram:
\begin{align*}
\xymatrix{ 
0 \ar[r] & U \ar[r]\ar[d] & V\ar[d]^{\phi}\ar[r]^{\pi}& W \ar[r] \ar[d]^{\rho} & 0 \\
 0 \ar[r] & 0 \ar[r] & V_x \ar[r]^{\pi_x}_{\sim}& W_{x} \ar[r] & 0  
}
\end{align*}
and denote by $\alpha \colon W \hookrightarrow V$ the natural inclusion, which satisfies $\pi \circ \alpha =\Id$. The composition $\alpha \circ \rho^{-1}\circ \pi_x$ is a section of $\phi.$ Indeed, as $\phi=\pi_x^{-1}\circ \rho \circ \pi,$ one has that 
$$\phi\circ (\alpha \circ \rho^{-1}\circ \pi_x)=\pi_x^{-1}\circ \rho\circ \pi \circ \alpha \circ \rho^{-1}\circ \pi_x=\Id_{V_{x}}$$
and we conclude that $U=\Ker(\phi)$ and that $W\simeq \rho(W)\simeq V_x$. In particular, we have that the localization map $V\to V_x$ admits a section, which is exactly $\sigma=\alpha\circ\rho^{-1}\circ \pi_x.$

Finally, uniqueness up to isomorphism as abstract spaces follows from the previous commutative diagram, as well as uniqueness of  $U$ as a subspace of $W$. There only remains to show that $W$ is unique as a subspace of $V,$ i.e. that the decomposition is unique up to equality of subspaces of $V$, or in other words, that the section $\sigma$ is unique.  To see this, we need to prove that there is only one supplementary $W\subseteq V$ of $U$ where $\f$ is an isomorphism. Consider $U\oplus W=U\oplus W'$ two different $\f$-invariant decompositions with $\f_{|W}$ and $\f_{|W'}$ being isomorphisms and let $w\in W$ be any vector. Then $\f(w)=u+w'$ with $u\in U$ and $w'\in W'$ by the decomposition, but $u=0$ by the $\f$-invariance; and since $\f_{|W'}$ is an isomorphism we have that $w=\f^{-1}(w')\in W'$. This shows that $W\subseteq W'$ and one argues analogously for the converse. 
\end{proof}
\end{thm}

\begin{cor}\label{C: pwCN}
Assuming that $\phi\colon V\to V_x$ is surjective, then $\phi$ admits a (unique) section of $k[x]$-modules if and only if $\f$ can be written as a sum $\f=\f_1+\f_2$ with $\f_1, \f_2\in\ed_k(V)$ satisfying that \begin{itemize}
\item $i(\f_1)\leq 1$;
\item $\f_2$ is pointwise nilpotent;
\item $\f_1\circ \f_2=0=\f_2\circ \f_1$. 
\end{itemize}
\begin{proof}

Let us assume that $\phi$ is surjective and that $\f=\f_1+\f_2$ as in the statement, then we can define $U=\Ker(\f_1)$ and $W=\Img(\f_1)$. Since $i(\f_1)\leq 1$, then $V=U\oplus W$ as $\f_1$-invariant subspaces. Let us see that this decomposition is also $\f$-invariant. Indeed, if $u\in U$, then $\f(u)=\f_2(u)$ by construction, so $\f_1(\f_2(u))=0$ by the third condition, $\f(u)=\f_2(u)\in\Ker(\f_1)$; on the other hand, if $w=\f_1(v)$, then $\f(w)=\f(\f_1(v))=\f_1(\f_1(v))\in \Img(\f_1)$. By the converse implication of Theorem \ref{T: general}, the section exists.  For the other implication, simply define 
$$\f_1(v)=\begin{cases}\f(v)\text{ if }v\in W\\
0\,\,\,\,\,\,\,\,\,\text{ if }v\in U\end{cases},\,\,\,\,\,\f_2(v)=\begin{cases}0\,\,\,\,\,\,\,\,\,\text{ if }v\in W\\
\f(v)\text{ if }v\in U\end{cases};$$
it is clear that these linear maps satisfy the conditions of the statement. 
\end{proof} 
\end{cor}
\begin{defn}\label{D: pw-CN}
We say that an endomorphism $\f$ is \textit{pointwise core-nipotent} when it an be written as a sum $\f=\f_1+\f_2$ with $\f_1, \f_2\in\ed_k(V)$ satisfying that \begin{itemize}
\item $i(\f_1)\leq 1$;
\item $\f_2$ is pointwise nilpotent;
\item $\f_1\circ \f_2=0=\f_2\circ \f_1$. 
\end{itemize}
\end{defn}

It is worth noting that the condition for surjectivity of $\phi\colon V\to V_x$ is weaker than stabilization of the chains \eqref{eq 1a} or \eqref{eq 2a}.
\begin{lem}\label{L: surjective}
The map $\phi\colon V\to V_x$ is surjective if and only if for each $v\in V$ there exists some $m\geq 0$ (dependent on $v$) such that $\f^m(v)\in \Img(\f^{m+s})$ for all $s\geq 0$.
\begin{proof}
Clearly, $\phi$ is surjective if each $v/x^s\in V_x$ is of the form $w/1$ for some $w\in V$. This happens if and only if $\f^m(v)=\f^{m+s}(w)$ for some $m$ dependent on $v$.
\end{proof}
\end{lem}

\begin{exam}
Let $k[[x]]$ be the $k$-vector space of formal power series and let $$D\colon k[[x]]\to k[[x]]$$
be the formal derivative operator, this is, $D(\sum_{i\geq 0} a_ix^i)=\sum_{i\geq 1}ia_i x^{i-1}$. The morphism
$$\phi_D\colon k[[x]]\to k[[x]]_x$$
is surjective by the previous Lemma, because every formal power series is integrable, and $\Ker(\phi_D)=k[x]$; i.e. we have an exact sequence
$$0\to k[x]\hookrightarrow k[[x]]\to k[[x]]_x\to 0.$$
This exact sequence of $k[x]$-modules is not split. 
\end{exam}

\begin{cor}
If the chain of inclusions \eqref{eq 2a} stabilizes, $\phi\colon V\to V_x$ is surjective. In that case, the morphism $\phi$ is naturally isomorphic to $\varphi^n\colon V\to\Img(\varphi^n)$ for any $n$ larger than the stabilization index of the chain. 
\begin{proof}
Under this hypothesis, the natural map $\Img(\varphi^n)\to \varinjlim_{i\geq 0}\Img(\varphi^i)\simeq V_x$ is an isomorphism. 
\end{proof}
\end{cor}
\begin{cor}
Let $\varphi\in\mathrm{End}_k(V)$ be an injective endomorphism. Then, it is an isomorphism (i.e. surjective) if and only if for each $v\in V$ there exists some $m\geq 0$ (dependent on $v$) such that $\f^m(v)\in \Img(\f^{m+s})$ for all $s\geq 0$.
\begin{proof}
If follows from Proposition \ref{P: Nuc y Loc}, Lemma \ref{L: surjective} and the fact that $\f$ is an isomorphism if and only if $\phi\colon V\to V_x$ is an isomorphism. It also follows from direct computation for $s=1$. 
\end{proof}
\end{cor}

\section{The ``pointwise Drazin inverse''}
If the exact sequence of $k[x]$-modules $0\to \Ker(\phi)\to V\to V_x\to 0$ (in the hypothesis of Theorem \ref{T: general}) splits, we have a definition for a generalized inverse of $\varphi$. Later, we will see that, when $\varphi$ is a core-nilpotent endomorphism, this natural inverse will coincide with the Drazin inverse of $\varphi$. 

The construction goes as follows: consider the commutative diagram
\begin{align*}
\xymatrix{ 
V\ar[d]_{\varphi}\ar[r]^{\phi}& V_x\ar[r]\ar[d]^{\varphi_x}& 0\\
V\ar[r]^{\phi} & V_x \ar[r] & 0.
}
\end{align*}
Since $\varphi_x$ is an isomorphism and $\sigma\colon V_x\to V$ is a section, we can define
\begin{align}\label{eq gen drazin}
\varphi^d:=\sigma\circ \varphi_x^{-1}\circ \phi.
\end{align}
Note that the fact that $\sigma$ is a section of $k[x]$-modules implies that 
$$\varphi\circ\sigma=\sigma\circ\varphi_x.$$
\begin{lem}
The endomorphism $\varphi^d$ satisfies the following properties:
\begin{itemize}
\item $\varphi^d\circ \varphi\circ \varphi^d=\varphi^d$;
\item $\varphi^d\circ \varphi=\varphi\circ \varphi^d$ is the projector onto $W$ induced by the section $\sigma;$
\item $\phi\circ\varphi\circ\varphi^d=\phi$.
\end{itemize}
\begin{proof}
For the first property, write
\begin{align*}\varphi^d\circ &\varphi\circ \varphi^d=\sigma\circ \varphi_x^{-1}\circ \phi\circ\varphi\circ \sigma\circ \varphi_x^{-1}\circ \phi=\\&=\sigma\circ \varphi_x^{-1}\circ \phi\circ\sigma\circ\varphi_x\circ \varphi_x^{-1}\circ \phi=\sigma\circ\varphi_x^{-1}\phi=\varphi^d;\end{align*}
where in the last step we have used that $\phi\circ\sigma=\Id_{V_x}$. The second property follows from the fact that all morphisms in the definition of $\varphi^d$ are of $k[x]$-modules and that $\varphi\circ\varphi^d=\varphi\circ\sigma\circ\varphi_x^{-1}\circ\phi=\sigma\circ\varphi_x\circ\varphi_x^{-1}\circ\phi=\sigma\circ\phi$, which is exactly the projector associated to $\sigma$.  As for the third property
$$\phi\circ\varphi\circ\varphi^d=\phi\circ\varphi\circ\sigma\circ\varphi_x^{-1}\circ\phi=\phi\circ\sigma\circ \varphi_x\circ\varphi_x^{-1}\circ\phi=\phi,$$
which completes the proof.
\end{proof}
\end{lem}
\begin{rem}
Explicitly, the third condition of the Lemma says that, for every $v\in V$, there exists some $m\geq 0$ dependent on $v$ such that $\varphi^d(\varphi^{m+1}(v))=\varphi^m(v)$. 
\end{rem}

Existence of such an ``inverse'' characterizes when the localization sequence splits.
\begin{lem}\cite[Definition 3.1]{Fpa-CN}\label{L: index 1}
If $f, g$ are endomorphisms such that $f=f\circ g \circ f$ and so that $f\circ g=g\circ f$, then $i(f)\leq 1$. 
\end{lem}

\begin{thm}[\textbf{Existence of the splitting}]\label{T: local drazin}
Let $\varphi$ be an endomorphism. The sequence $$0\to \Ker(\phi)\to V\overset{\phi}{\to} V_x\to 0$$ is split exact if and only if there exists an endomorphism $\alpha\colon V\to V$ such that 
\begin{itemize}
\item $\alpha\circ\varphi\circ \alpha=\alpha$;
\item $\alpha\circ\varphi=\varphi\circ\alpha$;
\item $\phi\circ\varphi\circ\alpha=\phi$. 
\end{itemize}
Furthermore, this endomorphism $\alpha$ is unique.
\begin{proof}
There only remains to prove the ``if'' part. By Corollary \ref{C: pwCN}, it suffices to check that $\varphi$ is pointwise core-nilpotent and, to be in the hypothesis of the aforementioned Corollary, to see that $\phi\colon V\to V_x$ is surjective. First, note that, by the first condition, $\alpha\circ\varphi=(\alpha\circ\varphi)^2$, so $\alpha\circ\varphi$ is a projector. By the third condition
$$\phi(\mathrm{Id}-\varphi\circ\alpha)=0,$$
hence $\Img(\mathrm{Id}-\varphi\circ\alpha)=\Ker(\alpha\circ \varphi)\subseteq\Ker(\phi)$. But we actually have an equality, because $v\in \Ker(\phi)$ if and only if $\varphi^m(v)=0$ for some $m\geq 0$ (dependent on $v$), and then 
$$(\alpha\circ\varphi)(v)=(\alpha\circ\varphi)^m(v)=(\alpha^m\circ\varphi^m)(v)=0.$$

Now, since we have shown that $\alpha\circ\varphi$ is a projector and have computed its kernel, we have an isomorphism of $k[x]$-modules $$\beta= (\mathrm{Id}-\alpha\circ\varphi, \alpha\circ\varphi)\colon V\to U\oplus W$$ with $W=\Img(\alpha\circ\varphi)=\Ker(\mathrm{Id}-\alpha\circ\varphi)$ and $U=\Ker(\phi)=\Ker(\alpha\circ\varphi)$. This allows us to write $\varphi=\varphi_1+\varphi_2$ with $\varphi_1$ being the extension by zero over $U$ of $\varphi_{|W}$ and $\varphi_2$ being the extension by zero over $W$ of $\varphi_{|U}$. This is the pointwise core-nilpotent decomposition of $\varphi$ because, on the one hand, $i(\varphi_{|W})\leq 1$ because $\varphi_1=\varphi\circ\alpha\circ\varphi$ by the projector condition and definition of $W$ and by Lemma \ref{L: index 1} together with the fact that 
$$(\varphi\circ\alpha\circ\varphi)\circ \alpha\circ (\varphi\circ\alpha\circ\varphi)=\varphi\circ\alpha\circ\varphi\circ \alpha\circ\varphi=\varphi\circ\alpha\circ\varphi.$$
On the other hand, $\varphi_2$ is pointwise nilpotent because the elements of $U$ are exactly those anihilated by some power of $\varphi$ that depends on the element. Finally, note that the condition
$$\phi(\mathrm{Id}-\varphi\circ\alpha)=(\mathrm{Id}-\varphi_x\circ\alpha_x)\phi=0$$
(where $\varphi_x, \alpha_x\colon V_x\to V_x$ are the induced morphisms) implies that  $$\Img(\phi)\subseteq\Ker(\mathrm{Id}-\varphi_x\circ\alpha_x)=\Img(\varphi_x\circ \alpha_x)\simeq (\Img(\varphi\circ\alpha))_x=W_x$$ (localization commutes with taking image because both spaces consist of vectors of the form $\phi(\alpha(v))/1$ for $v\in V$), so we have a natural morphism $$\Img(\phi)\hookrightarrow W_x\overset{\sim}{\leftarrow} W,$$
where the second isomorphism comes from the fact that $\varphi_{|W}$ is an isomorphism because $i(\varphi_1)\leq 1$ (otherwise $W=0$). We have a commutative diagram
$$\xymatrix{0\ar[r]&\Ker(\varphi\circ\alpha)\ar[r]\ar[d]&V\ar[r]^{\varphi\circ\alpha}\ar@{=}[d]& W\ar[r] & 0 \\
0 \ar[r] & \Ker(\phi) \ar[r] & V\ar[r]^{\phi} & \Img(\phi)\ar[u]\ar[r] & 0;}$$which, together with the fact that vertical inclusion on the left hand side is actually $\Ker(\phi)=\Ker(\varphi\circ\alpha)$ and by the Fifth Lemma, is enough to obtain that $\Img(\phi)\simeq W$.  Since the top exact sequence is split exact of $k[x]$-modules (with the section of $\varphi\circ\alpha$ provided by the natural inclusion), the bottom one is also split exact and $\Img(\phi)$ is indeed identified with $W$ via the section $\Img(\phi)\simeq W\subseteq V$. Again, since $\varphi_1$ is of index $\leq 1$, we have that $\varphi_{|W}$ is an isomorphism (otherwise $W=0$); and by Theorem \ref{T: general}, this shows that $\phi\colon V\to V_x$ must be surjective. Uniqueness of $\alpha$ now follows from the same result.
\end{proof}
\end{thm}
\begin{defn}
Let $\varphi\colon V\to V$ be an endomorphism. We say that $\alpha\colon V\to V$ is a \textit{pointwise Drazin inverse} of $\varphi$ if it satisfies the three conditions of Theorem \ref{T: local drazin}. By Theorem \ref{T: local drazin}, it is unique and we denote it by $\varphi^d$.
\end{defn}
\begin{cor}
An endomorphism $\varphi$ is pointwise core-nilpotent if and only if it admits a unique pointwise Drazin inverse. Furthermore, in this case, if $\varphi=\varphi_1+\varphi_2$ is the aforementioned decomposition:
\begin{itemize}
\item $\varphi_1=\varphi\circ\varphi^d\circ\varphi;$
\item $\varphi_2=\varphi-\varphi\circ\varphi^d\circ\varphi.$
\end{itemize}
\begin{proof}
This is all part of the proof of Theorem \ref{T: local drazin}.
\end{proof}
\end{cor}

\begin{exam}
Let $k[x]$ be the $k$-vector space of polynomials and consider the endomorphism $\varphi$ defined by
$$\varphi\bigg(\sum_{i\geq 0}a_ix^i\bigg)=\sum_{i\geq 0}(a_{2i+2}x^{2i}+(2i+2)a_{2i+1}x^{2i+1}).$$
 It is a check that 
$$\Ker(\phi_{\varphi})=\{\sum_{i\geq 0}a_i x^i: a_i=0\text{ for all }i\text{ odd}\}.$$
On the other hand, note that the map $\phi_{\varphi}\colon k[x]\to k[x]_x$ is surjective, because a polynomial $p(x)$ of degree $n$ satisfies that $\varphi^n(p(x))$ only has non-zero terms in even degree, hence it lives in the image of $\varphi^{n+s}$ for any $s\geq 0$. In particular, we have an exact sequence
$$0\to \Ker(\phi_{\varphi})\to k[x]\to k[x]_x\to 0.$$
This exact sequence is split and the pointwise Drazin inverse inducing the decomposition is 
$$\varphi^d(\sum_{i\geq 0}a_ix^i)=\sum_{i\geq 0}\dfrac{1}{2i+2}a_{2i+1}x^{2i+1}.$$			
Note that $\varphi$ is not core-nilpotent. However, the pointwise core-nilpotent decomposition of $\varphi$ is $\varphi=\varphi_1+\varphi_2$ with
\begin{itemize}
\item $\varphi_1\bigg(\sum_{i\geq 0}a_ix^i\bigg)=\sum_{i\geq 0}(2i+2)a_{2i+1}x^{2i+1};$
\item $\varphi_2\bigg(\sum_{i\geq 0}a_ix^i\bigg)=\sum_{i\geq 0}a_{2i+2}x^{2i}.$
\end{itemize}
\end{exam}

\section{Application to Core-Nilpotent Endomorphisms}
A sufficient condition for the section $\sigma$ of the previous part of this article to exist is $\varphi$ being a core-nilpotent endomorphism. In this case, $\varphi^d\equiv \varphi^D$ becomes  the Drazin inverse of $\varphi$ and the splitting of the localization exact squence yields the  AST decomposition of a core-nilpotent endomorphism (recall Definition \ref{D: CN}).

\begin{lem}\label{L: C-N}
If $\varphi \in \ed_k(V)$ is a core-nilpotent endomorphism of index $m$, the exact sequence $$0\to\Ker(\phi)\to V\to \Img(\phi)\to 0$$ is naturally isomorphic to 
$$0\to\Ker(\varphi^m)\to V\overset{\varphi^m}{\to} \Img(\varphi^m)\to 0$$
and, furthermore, it admits a unique section of $k[x]$-modules $\sigma\colon\Img(\varphi^m)\to V$
$$\sigma=j\circ ((\varphi^m)_{|\Img(\varphi^m)})^{-1}$$
where $j\colon \Img(\varphi^m)\hookrightarrow V$ is the natural inclusion and $(\varphi^m)_{|\Img(\varphi^m)}$ is
$$(\varphi^m)_{|\Img(\varphi^m)}\colon\Img(\varphi^m)\overset{\sim}{\to} \Img(\varphi^{2m})=\Img(\varphi^m).$$
\begin{proof}
The first part follows straightforwardly from the descriptions of $\Ker(\phi)$ and $\Img(\phi)$ given in Propositions \ref{P: Nuc y Loc} and \ref{P: Colim Img}. The $\sigma$ defined in the statement is clearly a section and a morphism of $k[x]$-modules because it is a composition of morphisms of $k[x]$-modules. Note that it is unique by Theorem \ref{T: general} (and its Corollary).
\end{proof}
\end{lem}
\begin{rem}
Explicitly on elements, the section $\sigma$ is
$$\sigma(\varphi^m(v))=\varphi^{m}(w)$$
with $w$ being some vector such that $\varphi^{2m}(w)=\varphi^m(v)$. Note that $w$ may not be unique, but $\varphi^m(w)$ is.
\end{rem}

\begin{prop}
If $\varphi \in \ed_k(V)$ is a core-nilpotent endomorphism of index $m$, then $\varphi^d$ satisfies
$$\varphi^{m+1}\circ\varphi^d=\varphi^m.$$
In particular, $\varphi^d$ is the Drazin inverse of $\varphi$, that we shall denote $\varphi^D$.
\begin{proof}
Indeed, under the identification $V_x\simeq \Img(\varphi^m),$ $\phi$ becomes isomorphic to $\varphi^m$, hence the third condition of Theorem \ref{T: local drazin} is exactly the condition of the statement.
\end{proof}
\end{prop}

We can summarize the results in the following statement:

\begin{thm}[\textbf{Characterization of CN-endomorphisms}]\label{T: CaractViaLocCN}
Let $\varphi\colon V\to V$ be an endomorphism. The following statements are equivalent:
\begin{itemize}
\item[I)] $\varphi$ is a Core-Nilpotent endomorphism (Definition \ref{D: CN}). 

\item[II)] The morphism $\phi\colon V\to V_x$ admits a (unique) section of $k[x]$-modules and $\varphi_{|\Ker(\phi)}$ is nilpotent.

\item[III)] There is a (unique) $\varphi$-invariant decomposition $V=U\oplus W$ with $\varphi_{|U}$ being nilpotent and $\varphi_{|W}$ being an isomorphism. 

\item[IV)] There is a (unique) decomposition $\varphi=\varphi_1+\varphi_2$ with $\varphi_2$ being nilpotent and $i(\varphi_1)\leq 1$ and such that $\varphi_1\circ\varphi_2=0=\varphi_2\circ\varphi_1$. 

\item[V)] $\varphi$ admits a (unique) Drazin inverse.
\end{itemize}
\begin{proof}
All the equivalences aside from II) are well-known from \cite{Fpa-CN}. It is clear from Lemma \ref{L: C-N} that I) implies II). We conclude by seeing that II) implies III). Indeed, the fact that $\phi$ admits a section implies that it is surjective and that the exact sequence 
$$0\to\Ker(\phi)\to V\to V_x\to 0$$
splits in a unique way with $V=U\oplus W$, where $U=\Ker(\phi)$ and $W$ is its corresponding supplementary submodule. By Theorem \ref{T: general}, $\varphi_{|W}$ is an isomorphism and, by hypothesis, $\varphi_{|U}$ is nilpotent. 
\end{proof}
\end{thm}




\medskip











 



\begin{thebibliography}{MMMM}




\bibitem{AST}  Argerami, M., Szechtman, F., Tifenbach, R.; \textit{On Tate's trace}, Linear Multilinear Algebra 55(6),515-520,(2007).


\bibitem{Ati} Atiyah, M.F., Macdonald, I.G.; \textit{Introduction to Commutative Algebra}, Addison-Wesley Publishing Company, (1969). 

\bibitem{Camp}  Campbell S.L.; \textit{The Drazin inverse of an infinite matrix}, SIAMJ Appl Math, 31, 492–503, (1976).

\bibitem{Dra} Drazin, M.P.; \textit{Pseudo-inverses in associative rings and semigroups}, Amer Math Monthly, 65(7), 506–514, (1958).

\bibitem{Gen CN} Karantha, M.P., Varkady, S.; \textit{ Generalized core-nilpotent decomposition}, J Anal 29, 543–550, (2021). 


\bibitem{Ind}  Mitra, S.K., Bhimasankaram, P., Malik, S.B.;\textit{Matrix Partial Orders, Shorted Operators and Applications}, World Scientific, (2010).


\bibitem{Pa-CN} Pablos Romo, F.; \textit{Core-Nilpotent Decomposition and new generalized inverses  of  Finite Potent Endomorphisms}, Linear Multilinear Algebra 68(11), 2254-2275, (2020).

\bibitem{Fpa-CN} Pablos Romo, F.; \textit{Core-Nilpotent Endomorphisms of Infinite-Dimensional Vector Spaces}, Mediterr. J. Math. 18(84), (2021). 

\bibitem{Graphs} Panelo, Cristian., Jaume, Daniel A.,  Machado Toledo, Maikon.; \textit{ On the Core-Nilpotent Decomposition of Threshold Graphs}, Panelo, Cristian and Jaume, Daniel A. and Machado Toledo, Maikon, SSRN: https://ssrn.com/abstract=5316214, (2025).


\bibitem{Ta}  Tate, J. \textit{Residues of Differentials on Curves}, Ann. Scient. \'Ec.
Norm. Sup. 1,4a s\'erie, 149-159,(1968).




\end{thebibliography}
\end{document}